\documentclass[12pt]{amsart}

\usepackage[margin=1.2in, centering]{geometry}

\usepackage{cite}

\DeclareMathOperator{\diam}{diam}
\DeclareMathOperator{\acosh}{arccosh}

\usepackage{enumerate}

\usepackage[utf8]{inputenc}
\usepackage{amssymb,amscd,amsmath,enumerate,amsthm}
\usepackage[all]{xy}
\usepackage[utf8]{inputenc}
\usepackage[english]{babel}
\usepackage{tikz}
\usetikzlibrary{positioning,chains,fit,shapes,calc}
\usepackage{mathtools}
\usepackage[colorlinks,linkcolor=blue]{hyperref}
\usepackage{ytableau}

\theoremstyle{plain}
\newtheorem{thm}{Theorem}[section]

\newtheorem{prop}[thm]{Proposition}
\newtheorem{lemma}[thm]{Lemma}

\theoremstyle{remark}
\newtheorem{remark}{\bf \quad \itshape  Remark}

\theoremstyle{plain}
\newtheorem{conjecture}{\bf \quad   Conjecture}

\theoremstyle{definition}

\newcommand{\R}{{\mathbb{R}}}

\newcommand{\F}{{\mathbb{F}}}
\newcommand{\E}{{\mathbb{E}}}

\newcommand{\Z}{{\mathbb{Z}}}

\newcommand{\bH}{{\mathbb{H}}}

\newcommand{\SL}{\mathrm{SL}}

\newcommand{\PSL}{\mathrm{PSL}}

\newcommand{\Alt}{{\raise 2pt\hbox{$\scriptstyle\bigwedge$}}}

\definecolor{myblue}{RGB}{80,80,160}
\definecolor{mygreen}{RGB}{80,160,80}
\newdimen\nodeSize
\newdimen\nodeDist
\tikzset{
	position/.style args={#1:#2 from #3}{
		at=(#3.#1), anchor=#1+180, shift=(#1:#2)
	}
}

\allowdisplaybreaks

\title{Erd\H{o}s distinct distances in hyperbolic surfaces}

\author{Zhipeng Lu}

\author{Xianchang Meng }
\address{Mathematisches Institut,
	Georg-August Universit\"{a}t G\"{o}ttingen,
	Bunsenstra{\ss}e 3-5,
	D-37073 G\"{o}ttingen,
	Germany}
	
\email{zhipeng.lu@uni-goettingen.de}
	
\email{xianchang.meng@uni-goettingen.de}


\date{}

\begin{document}

\maketitle

\begin{abstract}
    In this paper, we introduce the notion of ``geodesic cover" for Fuchsian groups, which summons copies of fundamental polygons in the hyperbolic plane to cover pairs of representatives realizing distances in the corresponding hyperbolic surface. Then we use estimates of geodesic-covering numbers to study the distinct distances problem in hyperbolic surfaces. Especially, for $Y$ from a large class of hyperbolic surfaces, we establish the nearly optimal bound $\geq c(Y)N/\log N$ for distinct distances determined by any $N$ points in $Y$,    where $c(Y)>0$ is some constant depending only on $Y$. In particular, for $Y$ being modular surface or standard regular of genus $g\geq 2$, we evaluate $c(Y)$ explicitly.
    We also derive  new sum-product type estimates.

\end{abstract}

\section{Introduction}
\subsection{Distinct distances problem in hyperbolic surfaces}
In 1946, Erd\H{o}s \cite{Erdos} posed the distinct distances problem which asks for the least number of distinct distances among any $N$ points in the Euclidean plane, and conjectured that it is $\sim N/\sqrt{\log N}$. Guth-Katz \cite{Guth-Katz} obtained the nearly optimal bound $\gtrsim N/\log N$ (we use the notation $f\gtrsim g$ to mean that there is an absolute constant $C>0$  with $f\geq Cg$). Erd\H{o}s also considered the higher dimensional generalization of the problem in $\R^d$ ($d\geq 3$) and conjectured the lower bound  $\gtrsim N^{2/d}$. For $d\geq 3$, Solymosi-Vu \cite{Solymosi-Vu} obtained the lower bound $\gtrsim N^{2/d-2/d(d+2)}$ by an induction on the dimension with the best known lower bound in the plane at that time as the base case. Combining the Guth-Katz bound with the induction of Solymosi-Vu, one may improve the lower
bounds of Solymosi-Vu for higher dimensional Euclidean spaces. For example when $d=3$, it gives the lower bound $\gtrsim N^{3/5-\epsilon}$ for any $\epsilon>0$, see Sheffer \cite{Sheffer} for details. There is also a continuous analogue of the problem in geometric measure theory, the Falconer's conjecture, asking about the lower bound of Hausdorff dimension of the sets in $\R^d$ for which the difference set has positive Lebesgue measure. Interested readers may check \cite{Falconer}, \cite{Guth-Iosevich-Ou-Wang}, \cite{Iosevich} etc. 
In addition to the Euclidean space, Erd\H{o}s-Falconer type problems have also been studied in vector spaces over finite fields and other spaces, see e.g. Bourgain-Katz-Tao  \cite{Bourgain-Tao}, Iosevich-Rudnev \cite{Iosevich-Rudnev},  Hart-Iosevich-Koh-Rudnev \cite{HIKR}, Tao \cite{Tao}, Rudnev-Selig \cite{Rudnev-Selig}, and Sheffer-Zahl \cite{Sheffer-Zahl} etc.

In the present paper, we establish lower bounds of distinct distances problem for a large class of hyperbolic surfaces. Hyperbolic surfaces as quotients of the hyperbolic plane $\bH^2$ by the action of Fuchsian groups, are locally isometric to $\bH^2$. By homogeneity, a lower bound for the distinct distances in $\bH^2$ bounds that of any hyperbolic surface from above. However, as geodesics may be complicatedly folded by the quotient of a Fuchsian group, it is not clear whether the nearly optimal lower bound as of Guth-Katz \cite{Guth-Katz} still holds for general hyperbolic surfaces. By studying actions of Fuchsian groups relatively explicitly and excavating a general notion of ``geodesic cover", we establish
\begin{thm}\label{thm-hyperbolic surfaces}
Assume $Y$ is the modular surface or a surface whose fundamental group is co-compact as a Fuchsian group. A set of $N$ points in $Y$ determines $\geq c(Y)N/\log N$ distinct distances for some constant $c(Y)>0$ depending only on $Y$. 
\end{thm}

In order to deal with various hyperbolic surfaces, we propose the concept of \textit{geodesic-covering number} (see Section \ref{concept-geo-cover}) of a hyperbolic surface, which itself can be of independent interest. The finiteness of the geodesic-covering number implies such type of lower bound in the above theorem for distinct distances problem.

In particular for $Y_g$ being standard regular of genus $g\geq 2$, whose fundamental domain in the upper half plane $\bH^2$ can be chosen as a standard regular $4g$-gon, we are able to estimate $c(Y_g)$ explicitly and get the following theorem. 
\begin{thm}\label{thm-genus-g}
For $Y_g$ being standard regular of genus $g\geq 2$, the lower bound of distinct distances among any $N$ points in $Y_g$ is $\geq c\frac{N}{g^{18}(\log N+\log g)}$ for some absolute constant $c>0$.
\end{thm}

Here the asymptotic goes with respect to both $g$ and $N$, which is not trivial only when $N\gtrsim g^{18}$. One may wonder how many points there can be with pairwise equal distance in a surface of genus $g$. See Section \ref{section-equilateral-dim} for more discussions. 

To evaluate $c(Y_g)$ explicitly for standard regular $Y_g$, we rely on hyperbolic trigonometry and connect it with the hyperbolic circle problem. As a natural analogue of the Gauss circle problem in $\bH^2$, the hyperbolic circle problem asks for the asymptotics of $\#\{\gamma\in\Gamma: d_{\bH^2}(z_0, \gamma\cdot z_0)\leq Q  \}  $ for discrete subgroups $\Gamma\leq\PSL_2(\R)$ and $Q>0$. This problem and related generalizations have been widely studied by various mathematicians including Delsarte \cite{Delsarte}, Huber \cite{Huber, Huber-analytic}, Selberg \cite{Selberg}, Margulis \cite{Margulis}, Patterson \cite{Patterson}, Iwaniec \cite{Iwaniec}, Phillips-Rudnick \cite{Phillips-Rudnick}, Boca-Zaharescu \cite{Boca-Zaharescu}, Kontorovich \cite{Kontorovich} etc. For our applications, we need certain uniformity of lattice counting over surfaces of genus $g$, specifically in the form of \eqref{lattice-count-uniform-g} in proof of Proposition \ref{prop-regular-4g-K-Gamma}.

Here we briefly sketch the strategy of proving Theorem \ref{thm-hyperbolic surfaces}, which is a consequence of Theorem \ref{thm-geodesic-to-distinct distance} with Propositions \ref{prop-finite pieces of domains} and \ref{prop-distinct distance on modular surface}. For any surface $Y$ with universal cover $\bH^2$, its fundamental group is isomorphic to a Fuchsian group $\Gamma_Y\leq\PSL_2(\R)$. Note that $\Gamma_Y$ acts on $\bH^2$ by M\"{o}bius transformation, we have $Y\simeq \Gamma_Y\backslash\bH^2$ endowed the hyperbolic metric from $\bH^2$. For any points $p,q\in Y$, we pick two representatives (still denoted by $p,q$) in a fundamental domain $F$ of $\Gamma_Y$. Then $d_Y(p,q)=\min_{\gamma\in\Gamma_Y}d_{\bH^2}(p,\gamma\cdot q)$. We want to find a  subset $\Gamma_0\subset\Gamma_Y$ such that for any $p,q\in Y$, we have $d_Y(p,q)=d_{\bH^2}(p, \gamma\cdot q)$ for some $\gamma\in\Gamma_0$.
We call the patched region $\cup_{\gamma\in\Gamma_0}\gamma(F)$ a \textit{geodesic cover} of $Y$ and call the smallest $|\Gamma_0|$, denoted by $K_{Y}$ (or $K_{\Gamma_Y}$), the \textit{geodesic-covering number} of $Y$ (or $\Gamma_Y$). 

In Section \ref{section-hyperbolic surfaces} we show that co-compact Fuchsian groups have finite geodesic-covering number. If a Fuchsian group $\Gamma$ is co-compact, its fundamental domain is a closed region without ideal points as vertices. This is equivalent to that $\Gamma\backslash\bH^2$ has finite hyperbolic area and $\Gamma$ contains no parabolic elements, see Corollary 4.2.7 of \cite{Katok}. Especially closed hyperbolic surfaces of genus $g\geq 2$ belong to this case. Moreover, Proposition \ref{prop-regular-4g-K-Gamma} establishes  the estimate $K_{Y_g}\lesssim g^6$ for $Y_g$ being standard regular of genus $g\geq 2$.  
For groups which are not co-compact, we show by explicit analysis that the modular group has finite geodesic-covering number. More specifically, Proposition \ref{prop-distinct distance on modular surface} establishes the estimate $K_{\PSL_2(\mathbb{Z})}\leq 10$. 

Now given any $N$ points $P\subset Y$, if $K_Y<\infty$ we duplicate the points to be $\tilde{P}=\cup_{\gamma\in\Gamma_0}\gamma(P)\subset\bH^2$ on a geodesic cover $\cup_{\gamma\in\Gamma_0}\gamma(F)$ of $Y$ with $|\Gamma_0|=K_Y$. By definition, the distances among points of $P$ in $Y$ all belong to the distances among points of $\tilde{P}$ in $\bH^2$. However, we are not allowed to apply the lower bound for the hyperbolic plane to points of $\tilde{P}$ directly, since we have more number of points now and the inequality actually goes to wrong direction. Instead, we resort to counting of distance quadruples of $\tilde{P}\subset\bH^2$ to establish Theorem \ref{thm-hyperbolic surfaces}. See  Theorem \ref{thm-geodesic-to-distinct distance} for details.

\begin{remark}
For completeness we include the case of flat tori, i.e. $g=1$. We may similarly define $K_\Gamma$ for any discrete subgroup $\Gamma$ of the rigid motion group of $\R^2$. For flat tori which correspond to $\Gamma\simeq\Z^2$, we immediately see that $K_\Gamma<\infty$. Thus by the result of Guth-Katz \cite{Guth-Katz}, the number of distinct distances among $N$ points on any flat torus is $\gtrsim N/\log N$.
\end{remark}
\begin{remark}
There is also an analogue of the unit distance problem in hyperbolic surfaces. Borrowing the arguments from Section 7.6 of \cite{Guth} based on estimates of crossing numbers, one may establish the Spencer-Szemer\'{e}di-Trotter bound to $\bH^2$, i.e. the number of pairs with unit (or equal) distance among any $N$ points in $\bH^2$ is $\lesssim N^{4/3}$. It is also and direct implication of Pach-Sharir theorem \cite{Pach-Sharir} applied to 
hyperbolic circles. For any set of $N$ points in a hyperbolic surface $Y$ with $K_Y$ finite, we lift it to a set of $K_{Y}N$ points on a geodesic cover of $Y$. By Spencer-Szemer\'{e}di-Trotter one may bound the number of unit (or equal) distances among any $N$ points on $Y$ by $\lesssim (K_YN)^{4/3}$. In particular for standard regular surfaces $Y_g$ of genus $g\geq 2$, by Proposition \ref{prop-regular-4g-K-Gamma}, the upper bound becomes $\lesssim g^{8}N^{4/3}$.

\end{remark}

More generally,  we also derive a lower bound for the number of distinct distances between points of any two finite sets $P_1$ and $P_2$ in hyperbolic surfaces with finite geodesic-covering numbers.
\begin{thm}\label{thm-cross distances}
Let $P_1,P_2\subset Y$ be any finite sets in a hyperbolic surface $Y$ with finite geodesic-covering number. Then we have
\[\big|\{d_{Y}(p_1,p_2): p_1\in P_1,p_2\in P_2\}\big|\gtrsim_{Y}\frac{|P_1|^2|P_2|^2}{|P_1\cup P_2|^3\log|P_1\cup P_2|}.\]
\end{thm}

\begin{remark}When $P_1$ and $P_2$ are roughly the same size, this lower bound is sharp up to a factor of $\log$. \end{remark}

\subsection{Sharpness of Theorem \ref{thm-genus-g} and conjectures on geodesic-covering number} \label{section-equilateral-dim}
In order to analyze the sharpness of Theorem \ref{thm-genus-g}, we connect it with the equilateral dimension of hyperbolic surfaces. 
The equilateral dimension of a metric space is defined to be the maximal number of points with pairwise equal distance. For the simplest example the equilateral dimension of the Euclidean space $\E^d$ is always $d+1$. The equilateral dimensions of various spaces have been studied by Alon-Milman \cite{Alon-Milman}, Guy \cite{Guy}, Koolen \cite{Koolen} etc. 
We are not aware of any non-trivial bound of equilateral dimension on hyperbolic surfaces in literature. We observe that our results can be applied to the equilateral dimension problem on hyperbolic surfaces. And in converse, the results for equilateral dimensions could also help us to analyze the sharpness of Theorem \ref{thm-genus-g}.

We claim that Theorem \ref{thm-genus-g} implies equilateral dimension of standard regular surfaces $Y_g$ of genus $g$ is $\lesssim g^{18+\epsilon}$. Suppose to the contrary for infinitely many $g$, the surface $Y_g$ has equilateral dimension $\geq C g^{18+\epsilon}$ for some constant $C>0$. Then for each such $g$ there exists a set of $M_g=C g^{18+\epsilon}$ points in $Y_g$ with pairwise equal distance. Hence its number of distinct distances is $1$. On the other hand, by Theorem \ref{thm-genus-g}, the number of distinct distances for any set of $M_g$ points is $\gtrsim \frac{M_g}{g^{18} \log (gM_g)}\gtrsim g^{\epsilon} $ which would approach infinity as $g\rightarrow\infty$. Contradiction.

However, from another approach one may show that the equilateral dimension of $Y_g$ is actually $\lesssim g$. Suppose there are $N_g$ points in $Y_g$ with pairwise equal distance $r>0$. Choosing a fundamental domain $F$ of $Y_g$, we draw a circle of radius $r$ in $\bH^2$ centered at one representative of the $N_g$ points, say $p_0$. By definition, each point has a representative lying on the circle with distance at least $r$ from each other. We order these representatives by $p_i, i=1,\ldots, N_g-1$. For adjacent $p_i,p_j$, let $\alpha_{ij}$ be the smaller positive angle between geodesics connecting $p_0,p_i$ and $p_0,p_j$. By hyperbolic trigonometry, since $d_{\bH^2}(p_i, p_j)\geq r$, 
\[\sin(\alpha_{ij}/2)=\frac{\sinh(d_{\bH^2}(p_i,p_j)/2)}{\sinh(r)}\geq\frac{\sinh(r/2)}{\sinh(r)}=\frac{1}{2\cosh(r/2)}.\]
In the proof of Proposition \ref{prop-regular-4g-K-Gamma}, we get the upper bound $\cosh r\lesssim g^2$, hence $\alpha_{ij}\gtrsim 1/g$. This shows that $N_g\lesssim g$ and hence the equilateral dimension of $Y_g$ is $\lesssim g$. 

By the above analysis on equilateral dimensions, we see that the lower bound for the number of distinct distances among any $N$ points should be better than trivial in the range $g\lesssim N\lesssim g^{18+\epsilon}$. Therefore the factor of $g$ in Theorem \ref{thm-genus-g} is not sharp.

One possible approach to improve Theorem \ref{thm-genus-g} is trying to get a better bound for geodesic-covering number $K_{Y_g}$. One may modify the definition of geodesic cover a little bit, to choose a set $\Gamma_1\subset\Gamma_Y$ for a surface $Y$ such that for any $p,q\in Y$, \begin{equation}\label{equation-radical geodesic cover}d_Y(p,q)=\min_{\gamma_1,\gamma_2\in\Gamma_1}d_{\bH^2}(\gamma_1\cdot p,\gamma_2\cdot q)=\min_{\gamma_{1},\gamma_2\in\Gamma_1}d_{\bH^2}(p,\gamma_1^{-1}\gamma_2\cdot q),\end{equation}
for $p,q$ treated as representatives in some fundamental domain of $\Gamma_Y$. Then $\Gamma_0=\Gamma_1^{-1}\Gamma_1$ is a geodesic cover in the original definition and $|\Gamma_1|$ may be estimated as $\sim |\Gamma_0|^{1/2}$ in many cases. However this definition appears not as convenient for computation. We also observe that for rectangle tori, the four fundamental polygons around a vertex patched together gives a geodesic cover in the sense above. Thus we are tempted to conjecture that the fundamental polygons around one vertex may also work for the hyperbolic case. 
\begin{conjecture}
For standard regular surfaces $Y_g$ of genus $g\geq 2$, the geodesic-covering number $K_{Y_g}$ is $\lesssim g$.
\end{conjecture}

In addition, since for any finite index subgroup $\Gamma'$ of a Fuchsian group $\Gamma$, its fundamental domain is the union of finitely many fundamental domains of $\Gamma$. If $K_\Gamma$ is finite, one may expect that $K_{\Gamma'}$ is also finite. We further make the conjecture below. 
\begin{conjecture}\label{conj-arithmetic group}
For any subgroup $\Gamma\leq\PSL_2(\Z)$ of finite index, its geodesic-covering number is finite.
\end{conjecture}

There are more examples of Fuchsian groups with finite geodesic-covering number. For example, the translation group $$\left\{\begin{pmatrix}1&n\\0&1\end{pmatrix}: n\in\Z\right\}$$ has the strip $\{x+iy: 0<x\leq 1, y>0\}$ as fundamental domain and its geodesic-covering number is $\leq 3$. However it is an infinite index subgroup of $\PSL_2(\Z)$. Other simple examples include finite subgroups of $\PSL_2(\R)$. We further make the following more general conjecture..
\begin{conjecture}
For any discrete subgroup $\Gamma\leq \PSL_2(\R)$ whose fundamental domain has finitely many sides (geometrically finite), its geodesic-covering number is finite. 
\end{conjecture}
There are more questions that could be asked. For instances, if a Fuchsian group has finite geodesic-covering number, does any of its finite indexed subgroup also have finite geodesic-covering number? If true, then by Poincar\'{e}'s theorem, for Conjecture 3 we only need to focus on groups without elliptic elements. In general, how does geodesic-covering number relate to the signature of a Fuchsian group? Even more general, how does it relate to Fenchel-Nielsen coordinates in the Teichm\"{u}ller space?

\textbf{Notation.} Throughout this paper we use the notation $f\gtrsim g$ to mean that there is an absolute constant $C>0$ such that $f\geq Cg$, and we use $f'\lesssim g'$ to mean that $|f'|\leq C' g'$ for some absolute constant $C'>0$. We use $f\asymp g$ to mean that $f\lesssim g$ and also $f\gtrsim g$. 
\bigskip

{\bf Acknowledgements.} We would like to thank Harald Helfgott for helpful discussions.  We  thank Misha Rudnev for his  comments and Adam Sheffer for pointing out the improved lower bounds in higher dimensional Euclidean spaces, and thank the comments by Amitay Kamber. Both authors are partially supported by the Humboldt Professorship of Professor Helfgott.

\section{Geodesic-covering number and distinct distances}\label{section-hyperbolic surfaces}

We propose the concept of \textit{geodesic-covering number} for discrete subgroups of $\PSL_2(\R)$ then use its estimates to deal with the distinct distances problem in closed hyperbolic surfaces and the modular surface.  

\subsection{Geodesic-covering number}\label{concept-geo-cover}
In general, let $\Gamma\leq G=\PSL_2(\R)$ be a Fuchsian group, which acts on $\bH^2$ discontinuously. The discrete subgroup $\Gamma$ is of \textit{first kind} if it has finite co-volume, i.e. a fundamental domain of $\Gamma\backslash\bH^2$ has finite hyperbolic volume. In particular, surface groups and the modular group $\PSL_2(\Z)$ are all Fuchsian groups of first kind. 

Generally for any discrete subgroup $\Gamma\subset\PSL_2(\R)$, let $Y$ be the hyperbolic surface associated with $\Gamma$ and $F$ be the fundamental domain of $Y$, we propose the question of finding a subset $\Gamma_0\subset\Gamma$ such that
\begin{equation}\label{equation-geodesic cover}d_{Y}(p,q)=\min_{\gamma\in\Gamma_0}d_{\bH^2}(p,\gamma(q)),\ \forall p,q\in Y.\end{equation}
We call the patched region of fundamental domains $U=\cup_{\gamma\in\Gamma_0}\gamma(F)$ a \textit{geodesic cover} of $Y$. We say $U$ is \textit{minimal} if the cardinality of $\Gamma_0$ attains the minimal and $U\supset F$ (roughly $1\in\Gamma_0$) and we denote by $K_{\Gamma}$ the smallest $|\Gamma_0|$. We call it the \textit{geodesic-covering number} of $\Gamma$.

We expect the geodesic-covering number is finite for many discrete subgroups of $\PSL_2(\R)$.
\begin{prop}\label{prop-finite pieces of domains}
For any co-compact discrete subgroup $\Gamma\subset\PSL_2(\R)$, its geodesic-covering number $K_{\Gamma}$ is finite. 
\end{prop}
\begin{proof}
If $\Gamma$ is co-compact, we may choose a closed fundamental domain $F\subset\bH^2$ without ideal points as vertices. Denote $Y:=\Gamma\backslash\bH^2$. Then its diameter
\[\diam(Y):=\sup_{x,y\in F}\min_{\gamma\in\Gamma}d_{\mathbb{H}^2}(x, \gamma(y))\] 
is finite.
Let $U\subset\bH^2$ be
\[U:=\{z\in\bH^2\mid d_{\bH^2}(z,F)\leq \diam(Y)\}\]
and $\tilde{U}\supset U$ be 
\[\tilde{U}:=\cup\{\gamma(F)\mid \gamma\in\Gamma, \gamma(F)\cap U\neq\emptyset\}.\]
For any $p,q\in Y$, choose two representatives $x,y\in F$. We claim that the shortest geodesic connecting $x$ and $y$ in $\bH^2$ lies in $\tilde{U}$. In other words, there is some representative $y'\in\tilde{U}$ of $q$ such that $d_{\bH^2}(x,y')=d_Y(p,q)$. Otherwise, $d_{Y}(p,q)=d_{\bH^2}(x,\gamma(y))$ for some $\gamma\in\Gamma$ with $\gamma(F)\cap U=\emptyset$, so that $d_{\bH^2}(x,\gamma(y))> \diam(Y)\geq d_{Y}(p,q)$, a contradiction.

Now for each fundamental domain $F'\subset\tilde{U}$, we choose a $\gamma'\in\Gamma$ such that $F'=\gamma'(F)$. The set $\Gamma_0$ consisting of these isometries satisfies (\ref{equation-geodesic cover}). The number of fundamental domains $F'\subset\tilde{U}$ is finite and so $K_\Gamma\leq |\Gamma_0|<\infty$.
\end{proof}

Let $\bH^2$ be the hyperbolic plane and $G=\PSL_2(\R)$ be its isometry group which acts on $\bH^2$ by M\"{o}bius transformation: 
\[z\mapsto \gamma\cdot z=\dfrac{az+b}{cz+d}, \text{ for }\gamma=\begin{pmatrix}a&b\\c&d\end{pmatrix}\in\PSL_2(\R), z\in\bH^2.\] 
Let $P\subset\bH^2$ be a set of $N$ points and define the set of \textit{distance quadruples}
\begin{equation}\label{defn-Q(P)}
    Q(P):=\{(p_1,p_2; p_3,p_4)\in P^4: d(p_1,p_2)=d(p_3,p_4)\neq 0\},
\end{equation}
where $d(\cdot,\cdot)$ denotes the hyperbolic metric. 
Denote the distance set by
\[d(P):=\{d(p_1,p_2): p_1,p_2\in P\}.\]
Then we have a close relation between $d(P)$ and $Q(P)$ as follows. Suppose $d(P)=\{d_i: 1\leq i\leq m\}$ and $n_i$ is the number of pairs of points in $P$ with distance $d_i$. So $|Q(P)|=\sum_{i=1}^mn_i^2$. Since $\sum_{i=1}^mn_i=2\binom{N}{2}=N^2-N$, by Cauchy-Schwarz inequality we get
\[(N^2-N)^2=\left(\sum_{i=1}^mn_i\right)^2\leq\left(\sum_{i=1}^mn_i^2\right)m=|Q(P)||d(P)|.\]
Rearranging the inequality gives
\begin{equation}\label{equation-Cauchy-Schwarz}
|d(P)|\geq\dfrac{N^4-2N^3}{|Q(P)|}.
\end{equation}

Tao  gave an argument in his blog \cite{Tao} then later fulfilled by \cite{Rudnev-Selig} with further details using Klein quadric to derive the following result. 
\begin{lemma}\label{lem-Quadruple}
\begin{equation}
    |Q(P)|\lesssim N^3\log N.
\end{equation}
\end{lemma}
We are also able to prove this result by working explicitly with isometries of $\bH^2$.  Combining this Lemma with \eqref{equation-Cauchy-Schwarz},  one derives a lower bound for distinct distances in the hyperbolic plane. 

Now we  connect the  geodesic-covering number with  distinct distances problem on any hyperbolic surface $Y$ with corresponding fundamental group $\Gamma\subset G$. 
\begin{thm}\label{thm-geodesic-to-distinct distance}
Assume $Y$ is a hyperbolic surface with fundamental group $\Gamma$ and $K_{\Gamma}$ is finite. Then a set of $N$ points on $Y$ determines 
$$ \gtrsim \frac{N}{K^3_{\Gamma}\log(K_{\Gamma}N) }$$
distinct distances. 
\end{thm}
\begin{proof}

For any set $P$ of $N$ points on $Y$, we choose a minimal geodesic cover $\Gamma_0\subset\Gamma$ with $|\Gamma_0|=K_\Gamma$ such that \[d_{Y}(P):=\{d_{Y}(p,q):  p,q\in P\}\subset d_{\bH^2}(\cup_{\gamma\in\Gamma_0}\gamma(P)).\]
Then
\begin{align*}
    Q_{Y}(P):=&\{ (p_1, p_2; p_3, p_4)\in P^4: d_ Y(p_1, p_2)=d_Y(p_3, p_4)\neq 0\} \nonumber \\
\subset& Q(\cup_{\gamma\in\Gamma_0}\gamma(P)),
\end{align*} 
where $Q(P)$ is defined in \eqref{defn-Q(P)}. Since $|\cup_{\gamma\in\Gamma_0}\gamma(P))|\leq K_\Gamma|P|=K_\Gamma N$, by Lemma \ref{lem-Quadruple} we get  
\begin{equation}
   |Q_Y(P)|\leq |Q(\cup_{\gamma\in\Gamma_0}\gamma(P))| \lesssim (K_{\Gamma}N)^3 \log (K_{\Gamma}N). 
\end{equation} 
Similar to \eqref{equation-Cauchy-Schwarz}, by the Cauchy-Schwarz inequality, we have
$$|d_Y(P)|\geq \frac{N^4-2N^3}{|Q_Y(P)|} \gtrsim \frac{N}{K^3_{\Gamma}\log(K_{\Gamma}N)  } .$$
We get the desired lower bound.
\end{proof}

\section{Distinct distances in hyperbolic surfaces}
In this section,  we  give precise estimates for  geodesic-covering numbers of closed hyperbolic surfaces and the modular surface.
\subsection{Closed hyperbolic surfaces of genus \texorpdfstring{$g\geq 2$}{}}

In this subsection we deal with surface groups. Here a surface group $\Gamma_g$ is the fundamental group of a closed hyperbolic surface $Y_g$ of genus $g \ge 2$, with the following presentation
\[\Gamma_g:=\langle a_i,b_i: 1\leq i\leq g\rangle,\]
in which $a_i,b_i\in G$ satisfy $[a_1,b_1]\cdots[a_g,b_g]=1$. Here $[a_i,b_i]=a_ib_ia_i^{-1}b_i^{-1}$ is the commutator. Topologically, the generators $a_i,b_i$ represent the homotopy classes of closed geodesics on $Y_g$ and the unique relator is derived from the condition of gluing sides of a $4g$-gon in $\bH^2$ as a fundamental domain of $Y_g$. Note that for non-isometric closed surfaces of fixed genus $g$, the subgroups $\Gamma_g$ could be different. The moduli space of isometry classes of surfaces of genus $g$ is characterized by the Teichm\"{u}ller space $T(Y_g)\simeq\R^{6g-6}$.

Now for the standard regular surfaces, we estimate their geodesic-covering numbers concretely as follows.
\begin{prop}\label{prop-regular-4g-K-Gamma}
For the surface of genus $g$ with standard regular fundamental $4g$-gon of inner angle $\frac{\pi}{2g}$, we have $K_{Y_g}\lesssim g^6$. 
\end{prop}
\begin{proof}
Let $\Gamma_g\subset G$ be the corresponding surface group.
For a standard regular geodesic $4g$-gon $F\subset\bH^2$ centered at $i$ (denote by $O$) serving as a fundamental domain of $Y_g$, we estimate its diameter as follows.

\begin{figure}[ht]
\begin{center}
\includegraphics[width=0.5\textwidth]{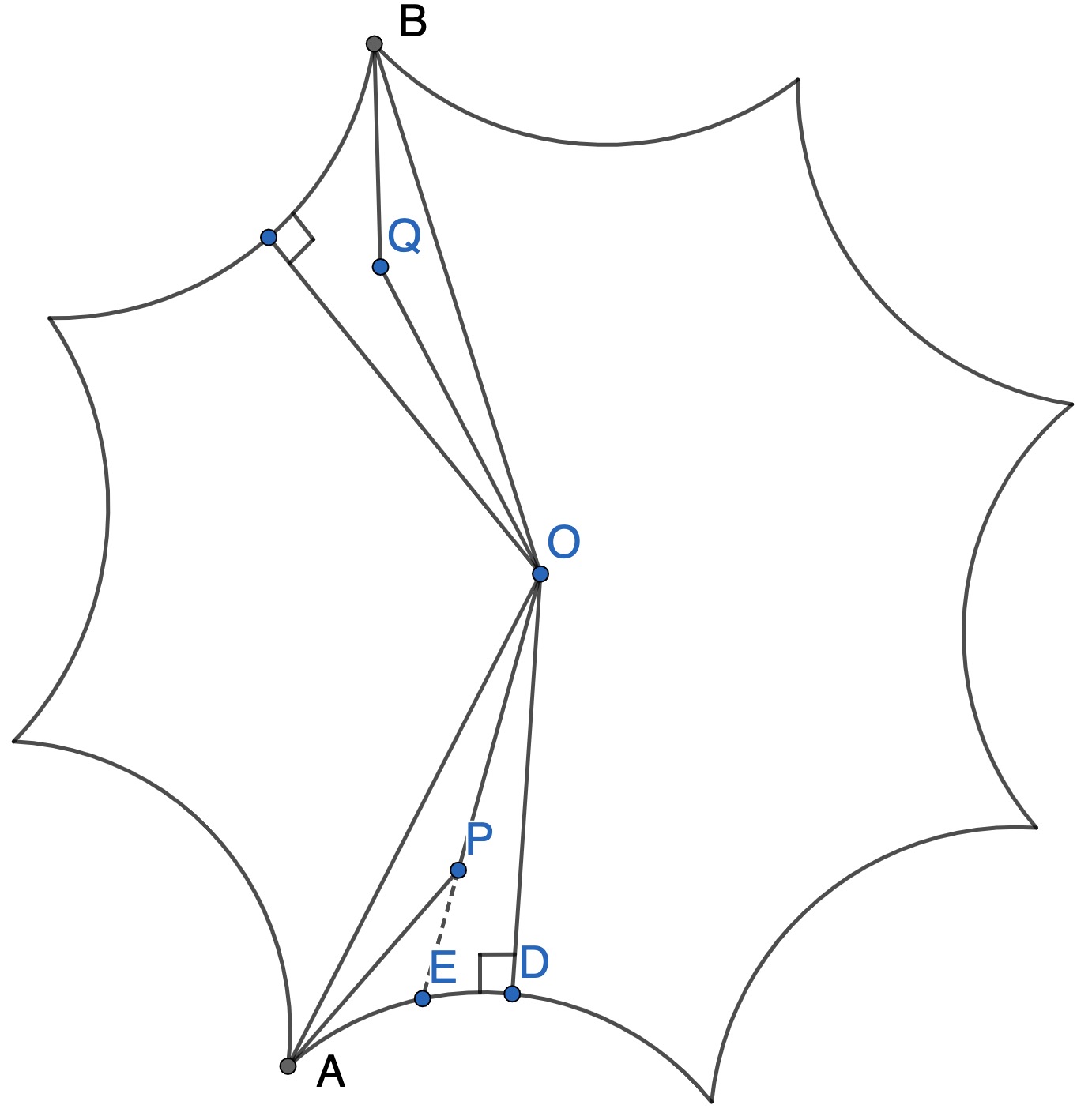}\\
\caption{Distance between $P$ and $Q$}
\label{polygon}
\end{center}
\end{figure}

First we determine a bound for $\diam(Y_g)$. For any $P, Q\in F$ (see Figure \ref{polygon}) , choose two vertices $A$ and $B$ of $F$ that are closest to $P$ and $Q$ correspondingly. Since there exists  $\gamma\in\Gamma_g$ such that $\gamma(A)=B$, we have ($d=d_{\mathbb{H}^2}$) by triangle inequality
\[d_{Y_g}(P,Q)\leq d(O,Q)+d(O,P),\ d_{Y_g}(P,Q)\leq d(P,A)+d(Q,B),\]
hence 
\[2d_{Y_g}(P,Q)\leq d(O,P)+d(P,A)+d(O,Q)+d(Q,B).\]
We claim that $d(O,P)+d(P,A)\leq d(O,D)+d(A,D)$. Indeed, if we extend the geodesic between $O$ and $P$ to $E$, by triangle inequality we have $d(O,E)=d(O,P)+d(P,E)\leq d(O,D)+d(D,E)$ and $d(P,A)\leq d(P,E)+d(E,A)$, so that
\begin{align*}
d(O,P)+d(P,A)&=d(O, E)-d(P, E)+d(P, A)\\
&\leq d(O,D)+d(D,E)-d(P,E)+d(P,E)+d(E,A)\\
&=d(O,D)+d(D,E)+d(E,A)\\
&=d(O,D)+d(A,D).
\end{align*}
Since $F$ is regular and $P,Q$ are arbitrary, we have
\[\diam(Y_g)\leq d(O,D)+d(A,D).\]
Note that  $(O,D,A)$ forms a right triangle and $\angle{AOD}=\angle{OAD}=\frac{\pi}{4g}=:\beta$, hyperbolic trigonometry gives
\begin{equation}\label{cosh-OD}
    \cosh(d(O,D))=\cosh(d(D, A))=\cot\beta,
\end{equation} thus
\[\cosh(\diam(Y_g))\leq \cosh(d(O,D)+d(D,A))=\cosh(2d(O,D))=2\cot^2(\beta)-1.\]
Also, we have 
\begin{equation}
    \cosh(d(O,A))=\cot^2(\beta).
\end{equation}

By construction of the geodesic cover $\tilde{U}$ as in the proof of Proposition \ref{prop-finite pieces of domains}, we only need to choose $\gamma\in\Gamma_g$ with $d(\gamma(i),i)\leq 2d(O,A)+\diam(Y_g)$. Since for any $\gamma=\begin{pmatrix}
a&b\\c&d\end{pmatrix}\in\PSL_2(\R)$,
\[2\cosh(d(\gamma(i),i))=\|\gamma\|^2=a^2+b^2+c^2+d^2,\]
we get
\begin{equation}\label{K-to-lattice}
K_{Y_g}\leq \#\{\gamma\in\Gamma_g: \|\gamma\|^2\leq 2\cosh(2d(O,A)+\diam(Y_g))\}.
\end{equation}
By the sum of arguments formula,
\begin{align}\label{4g-gon-radius+diam}
&\cosh(2d(O,A)+\diam(Y_g))\nonumber\\
&\quad=\cosh(2d(O,A))\cosh(\diam(Y_g))+\sinh(2d(O,A))\sinh(\diam(Y_g))\nonumber\\
&\quad=(2\cot^4(\beta)-1)(2\cot^2(\beta))-1)\nonumber\\
&\qquad+2\cot^2(\beta)\sqrt{\cot^4(\beta)-1}\sqrt{(2\cot^2(\beta)-1)^2-1}\nonumber\\
&\quad\lesssim g^4\cdot g^2+g^2\cdot g^2\cdot g^2\lesssim g^6.
\end{align}

By the result of counting hyperbolic lattices inside a circle (see \cite{Boca-Zaharescu} or \cite{Kontorovich}), we have asymptotically
\[\#\{\gamma\in\Gamma: \|\gamma\|\leq R\}\sim \frac{\pi}{\rm Area(\Gamma\backslash\bH^2)}R^2, ~\text{as}~R\rightarrow \infty.\]
This is the so-called \textit{hyperbolic circle problem} which has been widely studied in various prominent works. However, we cannot apply the above formula directly to \eqref{K-to-lattice} since we also need certain uniformity among the surfaces of genus $g$. 

Here we borrow an idea of Huber \cite{Huber-analytic} (sections 3.1 and 3.2) to show that, for $Y_g$ being standard regular with fundamental group $\Gamma_g$, 
\begin{equation}\label{lattice-count-uniform-g}
    N_{\Gamma_g}(R):=\#\{\gamma\in\Gamma_g: \|\gamma\|\leq R  \}\leq CR^2
\end{equation}
for some absolute constant $C$ independent of $g$. 

Now we prove \eqref{lattice-count-uniform-g}.  Since for any $\rm id\neq\gamma\in\Gamma_g$, it maps $F$ to another fundamental domain centered at $\gamma(i)$, the smallest distance between $i$ and $\gamma(i)$ is at least $2|OD|$ (see Figure \ref{polygon}). Let $\mathcal{D}(p, r):=\{x: d_{\bH^2}(x, p)<r \}$ be the disk of radius $r$ centered at $p$. Then the disks $\mathcal{D}(\gamma(i), |OD|)$
are disjoint for distinct $\gamma$. Thus we get 
$$\bigcup_{\|\gamma\|\leq R} \mathcal{D}(\gamma(i), |OD|)\subset \mathcal{D}(i, Q+|OD|), $$
where $Q=\acosh(R^2/2)$, which implies that
$${\rm Area}(\mathcal{D}(i, |OD|))\cdot N_{\Gamma_g}(R)=\sum_{\|\gamma\|\leq R} {\rm Area} (\mathcal{D}(\gamma(i), |OD|))\leq {\rm Area}(\mathcal{D}(i, Q+|OD|)).$$ 
By the hyperbolic area formula (\cite{Huber-analytic}, 2.10)
$${\rm Area}(\mathcal{D}(p, r))=2\pi (\cosh r-1),$$
together with \eqref{cosh-OD}, we get
$$ N_{\Gamma_g}(R)\leq \frac{\cosh(Q+|OD|)-1}{\cosh(|OD|)-1}\leq C_0 \cosh(Q)=\frac{C_0}{2}R^2   $$
for some absolute constant $C_0$ independent of $g$. 

Finally, applying \eqref{lattice-count-uniform-g} to \eqref{K-to-lattice} and by \eqref{4g-gon-radius+diam}, we deduce that 
$K_{Y_g}\lesssim g^6.$
We finish the proof.\end{proof}
\begin{remark}
Note that $\diam(Y_g)\geq d_{Y_g}(O, A)= d_{\bH^2}(O, A) $ in Figure \ref{polygon}, we have $$\cosh(\diam(Y_g))\geq \cosh(d(O, A))=\cot^2(\beta) \gtrsim g^2.$$
Thus by \eqref{4g-gon-radius+diam} the estimate of the diameter is optimal when applying the strategy of counting hyperbolic lattices. 
\end{remark}

\begin{remark}
One may also use the spectral decomposition and Weyl's law about the density of exceptional eigenvalues to get a uniform bound for the counting of hyperbolic lattices. 
\end{remark}

\subsection{Modular surface} The key to prove Proposition \ref{prop-finite pieces of domains} is that co-compact Fuchsian groups have fundamental domains of finite diameter. For other Fuchsian groups $\Gamma$ whose fundamental domain is not of finite diameter, the geodesic-covering number $K_\Gamma$ may still exist. In particular for the modular group $\PSL_2(\Z)$ we have the following result.
\begin{prop}\label{prop-distinct distance on modular surface}
For modular surface, we have $K_{\PSL_2(\Z)}\leq 10$ and the number of distinct distances among $N$ points on the modular surface $X$ is $\gtrsim N/\log N$.
\end{prop}
\begin{proof}
Let $F$ be the standard fundamental domain
\[F:=\{z\in\bH^2\mid -\dfrac{1}{2}<\Re(z)<\dfrac{1}{2}, |z|>1\}.\]
For any $z_1=x_1+y_1i,z_2=x_2+y_2i\in F$, it is immediately to verify the M\"{o}bius transformation
\[z_j=\begin{pmatrix}
\sqrt{y_j}&\dfrac{x_j}{\sqrt{y_j}}\\
0&\dfrac{1}{\sqrt{y_j}}\end{pmatrix}\cdot i=\gamma_j(i),\ j=1,2.
\]
Then for each $\gamma\in\SL_2(\Z)$ we get
\[2\cosh(d_{\bH^2}(z_1,\gamma(z_2)))=2\cosh(d_{\bH^2}(i,\gamma_1^{-1}\gamma\gamma_2(i)))=\|\gamma_1^{-1}\gamma\gamma_2\|^2.\]
Note that $\cosh(x)$ is monotonic for $x>0$, we see that
\[d_X(z_1,z_2)=\acosh(\min_{\gamma\in\SL_2(\Z)}\|\gamma_1^{-1}\gamma\gamma_2\|/2).\]
By computation,
\[\gamma_1^{-1}\gamma\gamma_2=\begin{pmatrix}\sqrt{\dfrac{y_2}{y_1}}(a-x_1c)&\dfrac{x_2a+b-x_1x_2c-x_1d}{\sqrt{y_1y_2}}\\
\sqrt{y_1y_2}c&\sqrt{\dfrac{y_1}{y_2}}(x_2c+d)
\end{pmatrix},\]
whence
\begin{align*}\|\gamma_1^{-1}\gamma\gamma_2\|^2=&\dfrac{y_2}{y_1}(a-x_1c)^2+\dfrac{1}{y_1y_2}(x_2a+b-x_1x_2c-x_1d)^2+y_1y_2c^2+\dfrac{y_1}{y_2}(x_2c+d)^2.
\end{align*}
If $c=0$, then $\gamma=\begin{pmatrix}\pm 1&b\\0&\pm 1\end{pmatrix}$ and
\[\|\gamma_1^{-1}\gamma\gamma_2\|^2=\dfrac{y_2}{y_1}+\dfrac{1}{y_1y_2}(x_2-x_1\pm b)^2+\dfrac{y_1}{y_2}.\]
Note that $-\dfrac{1}{2}<x_1,x_2<\dfrac{1}{2}$, the module $\|\gamma_1^{-1}\gamma\gamma_2\|^2$ attains minimum at $|b|\leq 1$ whose value is $<\dfrac{y_2}{y_1}+\dfrac{1}{4y_1y_2}+\dfrac{y_1}{y_2}:=U(0)$.

If $c\neq 0$, we have $\dfrac{a}{c}\dfrac{d}{c}-\dfrac{1}{c^2}=\dfrac{b}{c}$ and
\begin{align}\label{equation-c not 0}
    &\|\gamma_1^{-1}\gamma\gamma_2\|^2\\ \nonumber
   =&c^2\left[\dfrac{y_2}{y_1}\left(\dfrac{a}{c}-x_1\right)^2+\dfrac{1}{y_1y_2}\left(\dfrac{a}{c}x_2+\dfrac{b}{c}-x_1x_2-x_1\dfrac{d}{c}\right)^2+y_1y_2+\dfrac{y_1}{y_2}\left(\dfrac{d}{c}+x_2\right)^2\right]\\\notag
   =&c^2\left[\dfrac{y_2}{y_1}\left(\dfrac{a}{c}-x_1\right)^2+\dfrac{1}{y_1y_2}\left(\left(\dfrac{a}{c}-x_1\right)\left(\dfrac{d}{c}+x_2\right)-\dfrac{1}{c^2}\right)^2+y_1y_2+\dfrac{y_1}{y_2}\left(\dfrac{d}{c}+x_2\right)^2\right]\\\notag
   \geq& c^2y_1y_2\notag.
\end{align}
Comparing it with $U(0)$ and note that $-1/2<x_2,x_1<1/2$ we get
\begin{align*}\|\gamma_1^{-1}\gamma\gamma_2\|^2-U(0)&\geq c^2y_1y_2-\dfrac{y_2}{y_1}-\dfrac{1}{4y_1y_2}-\dfrac{y_1}{y_2}\\
 &> \dfrac{c^2y_1^2y_2^2-y_1^2-y_2^2-\dfrac{1}{4}}{y_1y_2}\\
 &=\dfrac{\left(|c|y_1^2-\dfrac{1}{|c|}\right)\left(|c|y_2^2-\dfrac{1}{|c|}\right)-\dfrac{1}{c^2}-\dfrac{1}{4}}{y_1y_2}.
\end{align*}
For $|c|\geq 2$ we have
\begin{align*}\|\gamma_1^{-1}\gamma\gamma_2\|^2-U(0)&\geq\dfrac{\left(2y_1^2-\dfrac{1}{2}\right)\left(2y_2^2-\dfrac{1}{2}\right)-\dfrac{1}{2}}{y_1y_2}\\
&>\dfrac{\left(\dfrac{3}{2}-\dfrac{1}{2}\right)^2-\dfrac{1}{2}}{y_1y_2}>0,\end{align*}
since $y_j>\sqrt{3}/{2}, j=1,2$. Thus in order to choose for $\Gamma_0$ as in (\ref{equation-geodesic cover}), we only need $\gamma\in\PSL_2(\Z)$ with $|c|\leq 1$.

For $|c|=1$, we have $ad\pm b=1$ (so that $a$ and $d$ can be chosen arbitrarily). We claim that in this case, (\ref{equation-c not 0})
attains minimum when $|a|\leq 1, |d|\leq 1$. By choosing $\gamma\in\SL_2(\Z)/\{\pm1\}$ we may assume $c=1$. Let $t_1=a-x_1$ and $t_2=d+x_2$, then (\ref{equation-c not 0}) becomes
\begin{equation}\label{equation-c=1}
\|\gamma_1^{-1}\gamma\gamma_2\|^2=\dfrac{y_2}{y_1}t_1^2+\dfrac{1}{y_1y_2}\left(t_1t_2-1\right)^2+\dfrac{y_1}{y_2}t_2^2+y_1y_2.
\end{equation}
We prove the claim by refuting the contradictory cases: (i) if $c=1, |a|\geq 2, |d|\geq 2$, note that $|x_1|\leq 1/2, |x_2|\leq 1/2$, then $|t_1|\geq 3/2, |t_2|\geq 3/2$ and (\ref{equation-c=1}) becomes
\[\|\gamma_1^{-1}\gamma\gamma_2\|^2\geq \dfrac{y_2}{y_1}\cdot\dfrac{9}{4}+\dfrac{1}{y_1y_2}\cdot\dfrac{25}{16}+\dfrac{y_1}{y_2}\cdot\dfrac{9}{4}+y_1y_2>U(0);\]
(ii) if $c=1, |a|\leq 1, |d|\geq 2$, then $|t_2|\geq 3/2$ and we take the difference ($t_1$ or $a$ fixed) 
\begin{align*}
    &\min_{c=1, |d|\geq 2}\|\gamma_1^{-1}\gamma\gamma_2\|^2-\min_{c=1, |d|\leq 1}\|\gamma_1^{-1}\gamma\gamma_2\|^2\\
    \geq&\dfrac{1}{y_1y_2}(|t_1|\cdot\dfrac{3}{2}-1)^2+\dfrac{y_1}{y_2}\cdot\dfrac{9}{4}-\min_{|d|\leq 1}\left\{\dfrac{1}{y_1y_2}(|t_1||d+x_2|-1)^2+\dfrac{y_1}{y_2}(d+x_2)^2\right\}\\
    \geq&\min_{|d|\leq 1}\left\{\dfrac{1}{y_1y_2}\left[\left(\dfrac{9}{4}-(d+x_2)^2\right)t_1^2+(2|d+x_2|-3)|t_1|\right]+\dfrac{y_1}{y_2}\cdot\left(\dfrac{9}{4}-(d+x_2)^2\right)\right\}\\
    \geq&\min_{|d|\leq 1}\left\{\dfrac{1}{y_1y_2}\left[2t_1^2+(2|d+x_2|-3)|t_1|+2y_1^2\right]\right\}\\
    \geq&\dfrac{1}{y_1y_2}\left[2\left(|t_1|-\dfrac{3}{4}\right)^2-\dfrac{9}{8}+2\cdot\dfrac{3}{4}\right]>0,
\end{align*}
noting that $|y_1|\geq\sqrt{3}/2$; (iii) for $c=1, |a|\geq 2, |d|\leq 1$, the above difference (for $t_2$ fixed) stays positive if symmetrically the roles of $t_1,d$ are replaced by $t_2,a$. Thus the claim is proved.

In conclusion, we may choose $\Gamma_0\subset\PSL_2(\Z)$ consisting of 
\[1,\begin{pmatrix}1&\pm 1\\0&1\end{pmatrix},\begin{pmatrix}1&0\\\pm1&1\end{pmatrix},\begin{pmatrix}0&-1\\1&\pm1\end{pmatrix},\begin{pmatrix}\pm1&- 1\\1&0\end{pmatrix},\begin{pmatrix}0&-1\\1&0\end{pmatrix}.\]
Thus $K_{\PSL_2(\Z)}\leq 10$. Then by Theorem \ref{thm-geodesic-to-distinct distance} we get the desired lower bound for distinct distances on modular surface. 
\end{proof}
Here the geodesic cover $\cup_{\gamma\in\Gamma_0}\gamma(\F)$ is $F$ together with the nine neighbouring fundamental domains on $\bH^2$. Actually we may only choose the geodesic cover in the sense of (\ref{equation-radical geodesic cover}) as $$\Gamma_1=\left\{1,\begin{pmatrix}1& 1\\0&1\end{pmatrix},\begin{pmatrix}1&0\\1&1\end{pmatrix},\begin{pmatrix}0&-1\\1&0\end{pmatrix}\right\}$$ 
since $\Gamma_1^{-1}\Gamma_1=\Gamma_0$.


\section{Distinct distances between two sets in hyperbolic surfaces}
This section contributes to the proof of Theorem \ref{thm-cross distances}. 

Let $P_1, P_2\subset Y$ be finite sets in a hyperbolic surface $Y$ with geodesic-covering number $K_Y<\infty$. Choose a geodesic cover $\Gamma_0\subset\Gamma_Y$ for the associated Fuchsian group $\Gamma_Y$ of $Y$ and $|\Gamma_0|=K_Y$, and duplicate $P_1,P_2$ to be $\tilde{P_j}=\cup_{\gamma\in\Gamma_0}\gamma\cdot P_j, j=1,2$. Define 
\[d_Y(P_1, P_2):=\{d_{Y}(p_1,p_2): p_1\in P_1, p_2\in P_2\},\]
\[Q_Y(P_1, P_2):=\{(p_1,p_2; q_1,q_2): p_1,q_1\in P_1, p_2,q_2\in P_2, d_{Y}(p_1,p_2)=d_{Y}(q_1,q_2)\neq 0\}.\]
and 
\[Q_{\bH^2}(\tilde{P_1}, \tilde{P_2}):=\{(p_1,p_2; q_1,q_2): p_1,q_1\in \tilde{P_1}, p_2,q_2\in \tilde{P_2}, d_{\bH^2}(p_1,p_2)=d_{\bH^2}(q_1,q_2)\neq 0\}.\]
Certainly $Q_Y(P_1, P_2)\subset Q_{\bH^2}(\tilde{P_1},\tilde{P_2})$.
Suppose $d_Y({P_1},{P_2})=\{d_1,\ldots,d_m\}$ and $n_k$ is the number of pairs $(p_1,p_2)$ for $p_1\in{P_1}, p_2\in {P_2}$ with $d_{Y}(p_1,p_2)=d_k$. We see that
$|{P_1}||{P_2}|-|P_1\cap {P_2}|=\sum_{k=1}^mn_k$ and $|Q_Y({P_1},{P_2})|=\sum_{k=1}^mn_k^2$. Then by the Cauchy-Schwarz inequality we get
\[|P_1|^2|P_2|^2\lesssim ( |{P_1}||{P_2}|-|{P_1}\cap {P_2}|)^2\leq  m\sum_{k=1}^mn_k^2=m|Q_Y({P_1},{P_2})|.\]
By Lemma \ref{lem-Quadruple}, we have
\[|Q_Y(P_1,P_2)|\lesssim |Q(\tilde{P_1}\cup\tilde{P_2})|\lesssim K^3_Y |P_1\cup P_2|^3\log (K_Y| P_1\cup P_2|),\]
where $Q(P)$ is defined in \eqref{defn-Q(P)}, and consequently
\[|d_Y(P_1,P_2)|\gtrsim_Y \dfrac{|P_1|^2|P_2|^2}{|P_1\cup P_2|^3\log|P_1\cup P_2|}.\]
This finishes the proof of Theorem \ref{thm-cross distances}.

We may replace $|P_1\cup P_2|$ by $\max\{|P_1|,|P_2|\}$ in the above inequality. If $|P_1|^2\leq|P_2|$, the inequality gives a trivial lower bound.

\end{document}